\documentclass{amsart}
\usepackage{amsmath,amsthm,amsfonts,amssymb,amscd,amsbsy,hyperref}
\usepackage{enumerate}
\usepackage{dsfont}

\newcommand{\dd}{\mathrm{d}}

\newcommand{\Vol}{\operatorname{Vol}}
\newcommand{\vol}{\operatorname{vol}}

\newcommand{\Iso}{\operatorname{Iso}}
\newcommand{\Ric}{\operatorname{Ric}}
\newcommand{\scal}{\operatorname{scal}}
\newcommand{\diam}{\operatorname{diam}}
\newcommand{\tra}{{\mathfrak t}}
\newcommand{\rot}{{\mathfrak r}}

\newcommand{\centr}{\operatorname{Z}}
\newcommand{\spec}{\operatorname{Spec}}
\newcommand{\aff}{\operatorname{Aff}(\R^d)}
\newcommand{\norm}{\operatorname{N}}

\renewcommand{\S}{\mathds S}
\renewcommand{\H}{\mathds H}

\newcommand{\N}{\mathds N}
\newcommand{\Z}{\mathds Z}
\newcommand{\R}{\mathds R}
\newcommand{\C}{\mathds C}

\newcommand{\GL}{\mathsf{GL}}
\newcommand{\SO}{\mathsf{SO}}
\renewcommand{\O}{\mathsf O}

\newcommand{\Sp}{\mathsf{Sp}}

\newcommand{\G}{\mathsf{G}}
\newcommand{\K}{\mathsf{K}}
\newcommand{\LL}{\mathsf{L}}

\newcommand{\Ad}{\operatorname{Ad}}

\newcommand{\g}{\mathrm g}
\renewcommand{\gg}{\mathbf g}
\newcommand{\h}{\mathrm h}

\newcommand{\y}{\operatorname Y}
\newcommand{\gr}{\g_\mathrm{round}}
\newcommand{\gf}{\g_\mathrm{flat}}
\newcommand{\ghyp}{\g_\mathrm{hyp}}
\newcommand{\gpr}{\g_\mathrm{prod}}

\allowdisplaybreaks

\hypersetup{
    pdftoolbar=true,
    pdfmenubar=true,
    pdffitwindow=false,
    pdfstartview={FitH},
    pdftitle={Infinitely many solutions to the Yamabe problem on noncompact manifolds},
    pdfauthor={Renato G. Bettiol and Paolo Piccione},
    pdfsubject={},
    pdfkeywords={}
    pdfnewwindow=true,
    colorlinks=true, 
    linkcolor=blue,
    citecolor=blue,
    urlcolor=black,
}

\newtheorem{theorem}{Theorem}[]
\newtheorem{lemma}[theorem]{Lemma}
\newtheorem{proposition}[theorem]{Proposition}
\newtheorem{corollary}[theorem]{Corollary}

\newtheorem{mainthm}{\sc Theorem}

\newtheorem{maincor}[mainthm]{\sc Corollary}

\newtheorem*{smlemma}{\sc Selberg-Malcev Lemma}
\newtheorem*{bieber}{\sc Bieberbach Theorems}

\theoremstyle{definition}

\theoremstyle{remark}
\newtheorem{remark}[theorem]{Remark}
\newtheorem{example}[theorem]{Example}

\title[Infinitely many solutions to the noncompact Yamabe problem]{Infinitely many solutions to the Yamabe problem on noncompact manifolds}

\author[R. G. Bettiol]{Renato G. Bettiol}
\author[P. Piccione]{Paolo Piccione}


\address{\begin{tabular}{lll}
University of Pennsylvania & & Universidade de S\~ao Paulo \\
Department of Mathematics & & Departamento de Matem\'atica \\
209 South 33rd St  & & Rua do Mat\~ao, 1010 \\
Philadelphia, PA, 19104-6395, USA & & S\~ao Paulo, SP, 05508-090, Brazil\\
\emph{E-mail address}: {\tt rbettiol@math.upenn.edu} & & \emph{E-mail address}: {\tt piccione@ime.usp.br}\\[0.3cm]
{\it Current address for R.\ G.\ Bettiol:} && \\
Max Planck Institute for Mathematics && \\
Vivatsgasse 7\\ 53111 Bonn, Germany &&
\end{tabular}
}

\allowdisplaybreaks
\numberwithin{equation}{section}
\numberwithin{theorem}{section}

\date{\today}

\begin{document}
\begin{abstract}
We establish the existence of infinitely many complete metrics with constant scalar curvature on prescribed conformal classes on certain noncompact product manifolds. These include products of closed manifolds with constant positive scalar curvature and simply-connected symmetric spaces of noncompact or Euclidean type; in particular, $\S^m\times\R^d$, $m\geq2$, $d\geq1$, and $\S^m\times\H^d$, $2\leq d<m$. As a consequence, we obtain infinitely many periodic solutions to the singular Yamabe problem on $\S^m\setminus\S^k$, for all $0\leq k<(m-2)/2$, the maximal range where nonuniqueness is possible. We also show that all Bieberbach groups in $\Iso(\R^d)$ are periods of bifurcating branches of solutions to the Yamabe problem on~$\S^m\times\R^d$, $m\geq2$, $d\geq1$.
\end{abstract}

\maketitle
\vspace{-0.22cm}

\section{Introduction}

The Yamabe problem on a Riemannian manifold $(M,\g)$ is to find a complete metric with constant scalar curvature which is conformal to $\g$. A landmark result in Geometric Analysis is that a solution always exists if $M$ is closed, see \cite{lee-parker} for a survey. 
The situation is much more delicate in the noncompact case, as there exist complete noncompact manifolds $(M,\g)$ for which the Yamabe problem does not have any solution~\cite{jin}. There are several partial existence results in the literature, such as \cite{aviles-mcowen,brs,grosse}, however existence is not settled in full generality.
In this paper, we exploit the geometry of discrete cocompact groups to provide large classes of noncompact manifolds on which the Yamabe problem has infinitely many \emph{periodic} solutions. 

We say that a solution to the Yamabe problem on a noncompact manifold $(M,\g)$ is \emph{periodic}, or \emph{$\Gamma$-periodic}, if it is the lift of a constant scalar curvature metric on a compact quotient $M/\Gamma$. The discrete cocompact group $\Gamma$ is the \emph{period} of the solution, in the sense that it is invariant under the action of $\Gamma$. In all instances studied in this paper, the infinitely many periodic solutions on noncompact manifolds correspond to infinitely many \emph{different periods}. In other words, these infinitely many metrics of constant scalar curvature do not descend to a common compact quotient, meaning that our multiplicity results are indeed noncompact phenomena.

Our first main result regards products of closed manifolds and symmetric spaces:

\begin{mainthm}\label{thm:first}
Let $(M,\g)$ be a closed manifold with constant positive scalar curvature, and $(N,\h)$ be a simply-connected symmetric space of noncompact or Euclidean type, such that the product $(M\times N,\g\oplus\h)$ has positive scalar curvature. Then there exist infinitely many periodic solutions to the Yamabe problem on $(M\times N,\g\oplus\h)$.
\end{mainthm}

An immediate consequence of Theorem~\ref{thm:first} is that there exist infinitely many periodic solutions to the Yamabe problem on $\S^m\times\H^d$ for all $2\leq d<m$, and on $\S^m\times\R^d$ for all $m\geq2$, $d\geq1$. We remark that there has been considerable interest in the Yamabe problem on such noncompact manifolds in recent years~\cite{akp,petean-henry,petean-ruiz}. Each of these infinitely many solutions is of the form $\phi\,\gpr$, where $\gpr$ is the corresponding product metric $\gr\oplus\ghyp$ or $\gr\oplus\gf$, and $\phi$ is a smooth positive function on $\S^m\times\H^d$ or $\S^m\times\R^d$ that does not depend on the $\S^m$ variable, by the asymptotic symmetry method of Caffarelli, Gidas and Spruck~\cite{cgs}.

It is easy to see that these infinitely many solutions on $\S^m\times\R$ translate into infinitely many solutions also on $\S^{m+1}\setminus\{\pm p\}$ and on $\R^{m+1}\setminus\{0\}$, which are conformally equivalent to $\S^m\times\R$ via the stereographic projection. A classification of these periodic solutions and their relation to solutions on the compact quotient $\S^m\times\S^1$ has been known for several years~\cite{kobayashi,schoen87}, for further details see also \cite{ccr}. These can be seen as simple instances of the so-called \emph{singular Yamabe problem}, which consists of finding solutions to the Yamabe problem on manifolds of the form $M\setminus\Lambda$, where $M$ is a closed manifold and $\Lambda\subset M$ a closed subset.

One of the most interesting consequences of Theorem~\ref{thm:first} regards a more involved instance of the singular Yamabe problem; that of the complement $\S^m\setminus\S^k$ of a round subsphere $\Lambda=\S^k$ in a round sphere $\S^m$. The special case $k=0$ is addressed above, using the stereographic projection on $\S^m\setminus\{\pm p\}$. For $k\geq1$, a direct computation shows that the (incomplete) round metric on $\S^m\setminus \S^k$ is conformally equivalent to the product $(\S^{m-k-1}\times\H^{k+1},\gpr)$, see Subsection~\ref{subsec:syp}. Thus, pulling back the infinitely many solutions we obtained in the latter yields the following:

\begin{maincor}\label{cor:syp}
There are infinitely many periodic solutions to the singular Yamabe problem on $\S^m\setminus\S^k$, for all $0\leq k<(m-2)/2$.
\end{maincor}

The above extends our previous result in \cite{bps-jdg}, where bifurcation techniques were used to obtain infinitely many solutions in the particular case of $\S^m\setminus\S^1$. These techniques cannot be used if $k>1$ due to the Mostow Rigidity Theorem. Furthermore, $0\leq k<(m-2)/2$ is the maximal range of dimensions for which multiplicity of periodic solutions is possible, by the asymptotic maximum principle.

Theorem~\ref{thm:first} is a particular case of the following more general multiplicity result:

\begin{mainthm}\label{thm:general}
Let $(M,\g)$ and $(\Sigma,\h)$ be closed Riemannian manifolds with constant scalar curvature, such that $\scal_\g>0$ and $\pi_1(\Sigma)$ has infinite profinite completion. Then there exists $\lambda_0>0$ such that, for any $\lambda>\lambda_0$, there
are infinitely many periodic solutions to the Yamabe problem on $\big(M\times\widetilde\Sigma,\,\g\oplus\lambda\,\widetilde\h\big)$.
\end{mainthm}

In the above, $(\widetilde\Sigma,\widetilde\h)$ denotes the Riemannian universal covering of $(\Sigma,\h)$, and $\lambda_0$ is the smallest nonnegative real number such that $\lambda_0\geq-\frac{\scal_\h}{\scal_\g}$. The profinite completion of a group is infinite if and only if there exists an infinite nested sequence of normal subgroups of finite index, see Subsection~\ref{subsec:groups}. For instance, infinite residually finite groups have infinite profinite completion.
In the proof of Theorem~\ref{thm:general}, this infinite chain of subgroups of $\pi_1(\Sigma)$ is used to produce an infinite chain of finite-sheeted coverings of $M\times\Sigma$ with arbitrarily large volume. Since the pull-back of $\g\oplus\lambda\,\h$ is not a Yamabe metric if one goes sufficiently high up along this chain, there must be another solution at some level. Iterating this argument gives the infinitely many solutions.

The first key input to prove Theorem~\ref{thm:first} using Theorem~\ref{thm:general} is a classical result of Borel~\cite{borel}, which states that every symmetric space $N$ of noncompact type admits irreducible compact quotients $\Sigma=N/\Gamma$. The same is obviously true for symmetric spaces of Euclidean type, that is, the Euclidean space $\R^d$. Second, if $(\Sigma,\h)$ is locally symmetric, then it clearly has constant scalar curvature, and $\pi_1(\Sigma)$ is infinite and residually finite (see Example~\ref{ex:infiniteresfinite}). Thus, $\pi_1(\Sigma)$ has infinite profinite completion. We may hence apply Theorem~\ref{thm:general}, which implies the desired statement in Theorem~\ref{thm:first}.
We stress that there are many manifolds $(\Sigma,\h)$ which are not locally symmetric but still satisfy the hypotheses of Theorem~\ref{thm:general}, see Subsection~\ref{subsec:examples}.

Despite providing infinitely many solutions to the Yamabe problem, none of the above results carries any information on the (local) arrangement of these solutions or the structure of their moduli space. However, this can be achieved through other techniques in the particular case of $M\times\R^d$, where $(M,\g)$ is a closed manifold with constant positive curvature. Namely, it can be shown that solutions \emph{bifurcate}, in the sense that there are sequences of new solutions forming branches that issue from a trivial $1$-parameter family of solutions.

In order to state our final main result, recall that a \emph{Bieberbach group} $\pi$ is a torsion-free crystallographic group, i.e., a discrete and cocompact subgroup of isometries of $\R^d$ that acts freely, so that $F=\R^d/\pi$ is a closed flat manifold. 

\begin{mainthm}\label{thm:flat}
Let $(M,\g)$ be a closed Riemannian manifold with constant positive scalar curvature and let $\pi$ be a Bieberbach group in the isometry group of $\R^d$, $d\geq2$. Then there exist infinitely many branches of $\pi$-periodic solutions to the Yamabe problem on $\big(M\times\R^d,\,\g\oplus\gf\big)$.
\end{mainthm}

An important part of the above statement is that \emph{every} Bieberbach group $\pi$ acting on $\R^d$ can be realized as the period of infinitely many periodic solutions to the Yamabe problem on $M\times\R^d$. The special case in which $\pi\cong\Z^d$ is a lattice follows from a recent result in \cite{hector}, using similar techniques.
These techniques to apply variational bifurcation theory to the Yamabe problem originated in \cite{bp-calcvar,bp-pacific,LPZ12}. The main input is a $1$-parameter family $\g_t$ of highly symmetric solutions that collapse at $t=0$. By proving that the Morse index of $\g_t$ becomes unbounded, we establish the existence of a sequence of bifurcation instants accumulating at $t=0$. The proof of Theorem~\ref{thm:flat} relies on showing that for all Bieberbach groups $\pi$ acting on $\R^d$, there exists such a collapsing $1$-parameter family $\h_t$ of flat metrics on $F=\R^d/\pi$, and understanding the spectral behavior of their Laplacian. The first task is achieved using results on the holonomy representation of $(F,\h_t)$, while the second follows from estimates relating the spectrum of the Laplacian on $(F,\h_t)$ with the diameter of this manifold.
Arbitrarily small eigenvalues of the Laplacian on $(F,\h_t)$ translate into arbitrarily large Morse index for $\g\oplus\h_t$ on $M\times F$, and hence bifurcation of this family of constant scalar curvature metrics on $M\times F$. Lifting these metrics to the universal covering $M\times\R^d$, we obtain the desired $\pi$-periodic solutions.

This paper is organized as follows. Section~\ref{sec:yamabe} is an overview of existence and uniqueness of solutions to the Yamabe problem on closed manifolds. Section~\ref{sec:multiplicity-coverings} begins with a discussion of closed manifolds whose fundamental group is residually finite or has infinite profinite completion, leading to the proof of Theorem~\ref{thm:general}, from which Theorem~\ref{thm:first} and Corollary~\ref{cor:syp} follow. Finally, Theorem~\ref{thm:flat} is proved in Section~\ref{sec:bifcollapseflat}, combining bifurcation theory with the collapse of closed flat manifolds.

\subsection*{Acknowledgements} It is a pleasure to thank Claude LeBrun for suggestions that eventually led to Theorem~\ref{thm:general}, Benson Farb and Matthew Stover for suggestions concerning Proposition~\ref{prop:fgsubgroup} and Example~\ref{ex:deligne}, and Andrzej Szczepa\'nski for discussions on the holonomy group of flat manifolds and reference \cite{hiss-szczepa}.

\section{Classical Yamabe Problem}\label{sec:yamabe}

In this section, we briefly recall some facts about the classical Yamabe problem for the convenience of the reader and to establish notation. A detailed exposition can be found in Lee and Parker~\cite{lee-parker}, Aubin~\cite[Chap.\ 5]{aubin-book}, or Schoen~\cite{schoen87}.

Given a closed Riemannian manifold $(M,\g_0)$, let $[\g_0]=\{\phi \, \g_0:\phi\in H^1(M)\}$ be the Sobolev $H^1$ conformal class of $\g_0$, and consider the Hilbert-Einstein functional
\begin{equation}\label{eq:afunct}
\mathcal A\colon [\g_0]\to\R, \quad \mathcal A(\g)=\Vol(M,\g)^{\frac{2-n}{n}}\int_M\scal_\g\vol_\g,
\end{equation}
where $n=\dim M$.
It is well-known that $\g\in[\g_0]$ is a critical point of the above functional if and only if $\scal_\g$ is constant, that is, $\g$ is a solution to the Yamabe problem (see \cite{lee-parker,schoen87}). In this case, the value of the functional is clearly
\begin{equation}\label{eq:aconst}
\mathcal A(\g)=\Vol(M,\g)^{\frac{2}{n}}\scal_\g.
\end{equation}
Existence of solutions is proved by showing that \eqref{eq:afunct} always achieves a minimum. More precisely, define the \emph{Yamabe invariant} of the conformal class $[\g_0]$ as
\begin{equation}\label{eq:yamabeinvariant}
\y\!\big(M,[\g_0]\big)=\inf_{\g\in [\g_0]}\mathcal A(\g).
\end{equation}
The combined work of Yamabe~\cite{yamabe}, Trudinger~\cite{trudinger68}, Aubin~\cite{aubin76}, and Schoen~\cite{schoen84} yields the following statement, that settled the \emph{existence} problem:

\begin{theorem}\label{thm:resyamabe}
There exists a metric $\g_\mathrm Y\in[\g_0]$, called \emph{Yamabe metric}, that achieves the infimum in \eqref{eq:yamabeinvariant}, $\y\!\big(M,[\g_0]\big)=\mathcal A(\g_\mathrm Y)$. Moreover, this minimum value satisfies
\begin{equation}
\y\!\big(M,[\g_0]\big)\leq\y\!\big(\S^n,[\gr]\big)
\end{equation}
and equality holds if and only if $(M,\g_0)$ is conformally equivalent to $(\S^n,\gr)$.
\end{theorem}

Regarding the \emph{uniqueness} problem, by the maximum principle, if $\y\!\big(M,[\g_0]\big)\leq0$, then $\g_\mathrm Y$ is the unique solution on $(M,\g_0)$. However, if $\y\!\big(M,[\g_0]\big)>0$, there may exist several metrics $\g\in[\g_0]$ with constant scalar curvature, but all have $\scal_\g>0$. Nonuniqueness phenomena for the classical Yamabe problem have been extensively studied in the literature, see, e.g., \cite{bm-jfa,bp-calcvar,bp-pacific,hebey-vaugon1,petean-asian,pollack,hector}.

\begin{remark}\label{rem:nonuniqueisom}
Nonuniqueness of solutions on~$(M,\g_0)$ means the existence of two or more 
nonconstant conformal factors $\phi\colon M\to\R$ such that $\phi\,\g_0$ has constant scalar curvature. However, it should be noted that these different conformal factors may give rise to isometric metrics.
For instance, consider the case of the round sphere $(\S^n,\gr)$, in which a metric $\g\in[\gr]$ has constant scalar curvature if and only if it is the pull-back of $\gr$ by a conformal diffeomorphism \cite[Sec.\ 2]{schoen87}. Thus, the \emph{moduli space} of solutions is the $(n+1)$-dimensional manifold $\mathrm{Conf}(\S^n,\gr)/\mathrm{Iso}(\S^n,\gr)\cong\SO(n+1,1)_0/\SO(n+1)$, but all the corresponding metrics are isometric to $\gr$.

Moreover, it is known that if $(M,\g_0)$ is a complete manifold not conformally diffeomorphic to $(\S^n,\gr)$ or $(\R^n,\gf)$, then there exists $\g_*\in[\g_0]$ such that $\mathrm{Conf}(M,\g_0)=\mathrm{Iso}(M,\g_*)$, see \cite{ferrand,schoenCR}. 
In particular, if $\g_0=\phi\,\g_*$ has constant scalar curvature and $f\in\mathrm{Conf}(M,\g_0)$ is nontrivial, then $f^*(\g_0)=(\phi\circ f)\,\g_*$ is isometric to $\g_0$ hence also has constant scalar curvature, and $\phi\circ f$ is a different conformal factor. However, by the same reasoning, if $\g_*$ itself has constant scalar curvature, then there are no other solutions on $[\g_*]$ that are isometric to $\g_*$.
\end{remark}

\section{Multiplicity of solutions via chains of coverings}
\label{sec:multiplicity-coverings}

We begin by discussing the relationship between infinite towers of finite-sheeted regular coverings of a manifold and the profinite completion of its fundamental group, providing several examples. This is then combined with Theorem~\ref{thm:resyamabe} to prove Theorem~\ref{thm:general}, from which Theorem~\ref{thm:first} and Corollary~\ref{cor:syp} follow.

\subsection{Profinite completion and residually finite groups}\label{subsec:groups}
A group $G$ is \emph{profinite} if it is isomorphic to the limit $\varprojlim G_s$ of some inverse system $\{G_s\}_{s\in S}$ of finite groups.\footnote{Although it is usual to consider profinite groups as \emph{topological groups}, assuming that $G_s$ have the discrete topology, for the purposes of this paper we consider them solely as algebraic objects.} For instance, the group of $p$-adic integers $\Z_p=\varprojlim\Z/p^n\Z$ is profinite.
Given a finitely generated group $G$, its \emph{profinite completion} is defined as the limit 
\begin{equation}
\widehat G=\varprojlim G/\Gamma,
\end{equation}
where $\Gamma$ runs over the collection of finite index normal subgroups of $G$.
Note that $\{G/\Gamma_j\}_{\Gamma_j\lhd G, \,[G:\Gamma_j]<\infty}$ is an inverse system, where $i>j$ corresponds to $\Gamma_i\subset\Gamma_j$, and the epimorphism $\phi_{ij}\colon G/\Gamma_i\to G/\Gamma_j$ is given by $\phi_{ij}(g\Gamma_i)=g\Gamma_j$.

Clearly $\widehat G$ is a profinite group, and it is characterized by the universal property that any group homomorphism $G\to H$, where $H$ is profinite, factors uniquely through a homomorphism $\widehat G\to H$. Furthermore, there is a natural homomorphism $\iota\colon G\to\widehat G$ induced by the projections, whose kernel is $\ker\iota=\bigcap_{\Gamma\lhd G, \,[G:\Gamma]<\infty} \Gamma$.

Groups $G$ for which $\iota$ is injective are called \emph{residually finite}. Equivalently, $G$ is residually finite if for any $g\in G\setminus\{e\}$, there exists a finite index subgroup $\Gamma\subset G$ such that $g\notin \Gamma$. Elementary arguments show that such $\Gamma$ may be assumed to be normal, as it can be replaced by $\mathrm{core}(\Gamma)=\bigcap_{h\in G} h\Gamma h^{-1}$, whose index satisfies $[G:\mathrm{core}(\Gamma)]\leq [G:\Gamma]!$. 
Key properties of residual finiteness are that if $G$ is residually finite, then so are all of its subgroups; and if $H$ is a finite index residually finite subgroup of $G$, then also $G$ is residually finite.

\subsection{Examples}\label{subsec:examples}
Let us mention a few classes of examples of finitely generated groups with infinite profinite completion, as well as some closed manifolds whose fundamental groups satisfy this property, which is a hypothesis in Theorem~\ref{thm:general}.

Our main source of examples is the class of infinite residually finite groups. Since their natural homomorphism $\iota\colon G\to\widehat G$ is injective, these groups trivially have infinite profinite completion. For example, every finitely generated abelian (or, more generally, nilpotent) group is residually finite. The following classical result provides a very rich family of finitely generated residually finite groups:

\begin{smlemma}
Finitely generated linear groups are residually finite.
\end{smlemma}

Recall that a group is \emph{linear} if it is isomorphic to a subgroup of $\GL(n,\C)$ for some $n\in\N$. A proof of the Selberg-Malcev Lemma can be found in~\cite[Sec.\ 7.6]{ratcliffe}.

\begin{example}\label{ex:infiniteresfinite}
The fundamental group of any locally symmetric space $\Sigma$ of noncompact type is residually finite (and infinite, if $\Sigma$ is closed).
Indeed, if $\Sigma$ is irreducible, then $\widetilde\Sigma=\G/\K$, where $\G$ is a semisimple noncompact Lie group, $\K$ is the maximal compact subgroup, and $\Sigma=\widetilde\Sigma/\Gamma$, where $\Gamma$ is a discrete torsion-free subgroup of $\G$. 
The image of $\Gamma$ under the adjoint representation $\Ad\colon \G\to\GL(\mathfrak g)$ is linear, and $\Ad(\Gamma)\cong\Gamma/\mathrm{Z}(\G)\cap\Gamma$. Assuming that the $\G$-action on $\G/\K$ is effective, $\mathrm{Z}(\G)$ is trivial, hence $\Gamma\cong\Ad(\Gamma)$ is residually finite as a consequence of the Selberg-Malcev Lemma. In particular, the fundamental group of any space form of nonpositive curvature is infinite and residually finite, hence has infinite profinite completion.
\end{example}

It follows from the proof of the Geometrization Conjecture that the fundamental group of any $3$-manifold is residually finite. For a topological viewpoint on residual finiteness of fundamental groups, see Reid~\cite[Sec.\ 2.1]{long-reid}. We remark that there exist \emph{nonlinear} finitely generated residually finite groups \cite{drusap}.

There also exist many finitely generated groups that are not residually finite but have infinite profinite completion, such as the \emph{Baumslag-Solitar groups}
\begin{equation*}
\mathrm{BS}(m,n)=\left\langle a,b : b\,a^m\, b^{-1}=a^n \right\rangle\!, \quad m>n>1.
\end{equation*}
Since this is a finitely presented group, it can be realized as the fundamental group of a closed manifold.
A useful tool to produce finitely generated groups with infinite profinite completion that are not necessarily residually finite is the following, communicated to us by B. Farb:

\begin{proposition}\label{prop:fgsubgroup}
If $\LL$ is a finitely generated infinite subgroup of a Lie group $\G$, then $\LL$ has infinite profinite completion.
\end{proposition}

\begin{proof}
Consider the image of $\LL$ under adjoint representation $\Ad\colon\G\to\GL(\mathfrak g)$. By the Selberg-Malcev Lemma, $\Ad(\LL)$ is residually finite. Thus, if $\Ad(\LL)$ is infinite, then so is $\iota(\Ad(\LL))\subset\widehat{\Ad(\LL)}$. By the universal property of $\LL$, since $\widehat{\Ad(\LL)}$ is profinite, the homomorphism $\iota\circ\Ad$ factors through a homomorphism $\widehat\LL\to\iota(\Ad(\LL))$, hence $\widehat\LL$ must also be infinite. Else, if $\Ad(\LL)$ is finite, then $\LL\cap\mathrm{Z}(\G)$ has finite index in $\LL$ and is hence a finitely generated abelian group. Thus, $\LL\cap\mathrm{Z}(\G)$ is residually finite, and hence so is $\LL$. This implies that $\LL\to\widehat\LL$ is injective, so $\widehat\LL$ is infinite.
\end{proof}

\begin{example}\label{ex:deligne}
According to Deligne~\cite{deligne}, the universal central extension of $\Sp(2n,\Z)$, which is the inverse image $\widetilde{\Sp(2n,\Z)}$ of $\Sp(2n,\Z)$ in the universal covering of $\Sp(2n,\R)$, is not residually finite for all $n\geq2$. However, $\widetilde{\Sp(2n,\Z)}$ is a lattice in a connected Lie group, hence it has infinite profinite completion by Proposition~\ref{prop:fgsubgroup}.
\end{example}

\begin{example}
Any group with positive first (rational) Betti number has infinite profinite completion.
\end{example}

\begin{remark}
It is generally difficult to exhibit finitely generated groups with \emph{finite} profinite completion, besides simple groups. One such example is the \emph{Higman group}
\begin{equation*}
\mathrm{Hig}=\left\langle a,b,c,d : a^{-1}\,b\,a=b^2,\ b^{-1}\,c\,b=c^2,\ c^{-1}\,d\,c=d^2,\ d^{-1}\,a\,d=a^2 \right\rangle\!,
\end{equation*}
which is infinite, finitely presented, and has no proper normal subgroups of finite index.
\end{remark}

\subsection{Coverings}
Let $\Sigma$ be a closed manifold with fundamental group $G=\pi_1(\Sigma)$. Recall that there is a natural bijective correspondence between conjugacy classes of subgroups of $G$ and equivalence classes of coverings of $\Sigma$. The trivial subgroup of $G$ corresponds to the universal covering $\widetilde\Sigma\to\Sigma$. A normal subgroup $\Gamma\lhd G$ of index $n=[G:\Gamma]$ corresponds to the $n$-sheeted regular covering $\widetilde\Sigma/\Gamma\to\Sigma$, where $\widetilde\Sigma/\Gamma$ is the quotient by the restriction to $\Gamma$ of the monodromy action of $G$ on $\widetilde\Sigma$. The group of deck transformations of this covering $\widetilde\Sigma/\Gamma\to\Sigma$ is $G/\Gamma$, and $\pi_1(\widetilde\Sigma/\Gamma)=\Gamma$.

\begin{lemma}\label{lemma:largevol}
Let $\Sigma$ be a closed manifold, $G=\pi_1(\Sigma)$. The following are equivalent:
\begin{enumerate}[{\rm (i)}]
\item $G$ has infinite profinite completion;
\item There exists an infinite nested sequence of normal subgroups $\Gamma_j\lhd G$,
\begin{equation*}
\dots\subsetneq\Gamma_j\subsetneq\dots\subsetneq\Gamma_2\subsetneq\Gamma_1\subsetneq G,
\end{equation*}
with finite index $n_j=[G:\Gamma_j]$, $2\leq n_j<\infty$;
\item For any $V>0$ and any Riemannian metric $\h$ on $\Sigma$, there is a finite-sheeted regular covering $\rho\colon\widetilde{\Sigma}/\Gamma\to\Sigma$ such that $\Vol(\widetilde\Sigma/\Gamma,\rho^*\h)>V$.
\end{enumerate}
\end{lemma}

\begin{proof}
The equivalence between (i) and (ii) follows from the definition of limit of an inverse system. The equivalence between (ii) and (iii) follows from the correspondence between normal subgroups $\Gamma\lhd G$ and regular coverings $\widetilde\Sigma/\Gamma\to\Sigma$ described above, using the fact that $\Vol(\widetilde\Sigma/\Gamma,\rho^*\h)=[G:\Gamma]\Vol(\Sigma,\h)$.
\end{proof}

\begin{remark}
The statements (ii) and (iii) in Lemma~\ref{lemma:largevol} remain equivalent to (i) even if we remove the words ``normal" from (ii) and ``regular" from (iii).
\end{remark}

\subsection{Multiplicity of solutions}
We now combine the above discussion of covering spaces with Theorem~\ref{thm:resyamabe} to prove Theorem~\ref{thm:general}.

\begin{proof}[Proof of Theorem~\ref{thm:general}]
Let $\lambda_0$ be the smallest nonnegative constant such that $\lambda_0\geq-\frac{\scal_\h}{\scal_\g}$. Fix $\lambda>\lambda_0$ and consider the product manifold $\big(M\times\Sigma,\g\oplus\lambda\,\h\big)$, which has constant positive scalar curvature.
We are going to define a sequence $\gg_j$ of metrics on finite-sheeted regular coverings of $M\times\Sigma$ with constant positive scalar curvature, such that the pull-backs of $\gg_j$ to $M\times\widetilde\Sigma$ provide the desired periodic solutions.

If $\mathcal A(\g\oplus\lambda\,\h)>\y\!\big(M\times\Sigma,[\g\oplus\lambda\,\h]\big)$, then let $\gg_1$ be a Yamabe metric in $[\g\oplus\lambda\,\h]$, see Theorem~\ref{thm:resyamabe}. Otherwise, if $\mathcal A(\g\oplus\lambda\,\h)=\y\!\big(M\times\Sigma,[\g\oplus\lambda\,\h]\big)$, i.e., $\g\oplus\lambda\,\h$ is a Yamabe metric, then by Lemma~\ref{lemma:largevol} there is a finite-sheeted regular covering $\Sigma_1\to\Sigma$ such that the product map $p_1\colon M\times\Sigma_1\to M\times\Sigma$ of the identity in $M$ and $\Sigma_1\to\Sigma$ satisfies
\begin{equation*}
\Vol\!\big(M\times\Sigma_1,\,p_1^*(\g\oplus\lambda\,\h)\big)^{\frac{2}{n}}\scal_{\g\oplus\lambda\,\h}>\y\!\big(\S^n,[\gr]\big).
\end{equation*}
By Theorem~\ref{thm:resyamabe}, there is a Yamabe metric $\gg_1$ in $[p_1^*(\g\oplus\lambda\,\h)]$ with constant positive scalar curvature. Note that $\gg_1$ is not isometric to $p_1^*(\g\oplus\lambda\,\h)$ since they lie in different levelsets of the functional $\mathcal A$ on this conformal class.

The fundamental group of $\Sigma_1$ is a finite index normal subgroup of $\pi_1(\Sigma)$, hence its profinite completion is also infinite.
 Applying Lemma~\ref{lemma:largevol} again, there is a finite-sheeted regular covering $\Sigma_2\to\Sigma_1$ such that the product map $p_2\colon M\times \Sigma_2\to M\times \Sigma_1$ of the identity in $M$ and $\Sigma_2\to\Sigma_1$ satisfies
\begin{equation*}
\Vol\!\big(M\times\Sigma_2,\,p_2^*(\gg_1)\big)^{\frac{2}{n}}\scal_{\gg_1}>\y\!\big(\S^n,[\gr]\big).
\end{equation*}
By Theorem~\ref{thm:resyamabe}, there is a Yamabe metric $\gg_2$ in $[p_2^*(\gg_1)]$ with constant positive scalar curvature. Once more, $\gg_2$ is not isometric to $p_2^*(\gg_1)$ since they have different values of $\mathcal A$.

Proceeding inductively in the above manner, we obtain an infinite sequence of finite-sheeted regular coverings
\begin{equation*}
\dots\longrightarrow\Sigma_j\longrightarrow\dots\longrightarrow\Sigma_2\longrightarrow\Sigma_1\longrightarrow\Sigma,
\end{equation*}
such that the maps $p_j\colon M\times\Sigma_j\to M\times\Sigma_{j-1}$ satisfy
\begin{equation*}
\Vol\!\big(M\times\Sigma_j,\,p_j^*(\gg_{j-1})\big)^{\frac{2}{n}}\scal_{\gg_{j-1}}>\y\!\big(\S^n,[\gr]\big),
\end{equation*}
and $\gg_j$ is a Yamabe metric in $[p_j^*(\gg_{j-1})]$, which has constant positive scalar curvature and is not isometric to $p_j^*(\gg_{j-1})$.
The pull-backs of $\gg_j$ to $M\times\widetilde\Sigma$ clearly lie in the conformal class of $\g\oplus\lambda\,\widetilde\h$ and correspond to pairwise different conformal factors, providing the desired infinitely many periodic solutions.
\end{proof}

\begin{remark}
The main technique in the above proof can be seen as an extension of some arguments of Hebey and Vaugon~\cite{hebey-vaugon1} to a more general class of manifolds.
\end{remark}

\begin{remark}\label{rem:nonisomthm}
Despite the fact that at each step in the above construction the new metric $\gg_j$ with constant scalar curvature on $M\times\Sigma_j$ is not isometric to the previous one $p_j^*(\gg_{j-1})$, in this level of generality, we cannot guarantee that their pull-backs to $M\times\widetilde\Sigma$ remain nonisometric. This corresponds to determining whether two distinct conformal factors can be obtained from one another by composition with a conformal diffeomorphism, see Remark~\ref{rem:nonuniqueisom}.

More information in this regard may be available in some particular cases, such as when $M\times\widetilde\Sigma$ is $\S^m\times \H^d$ with its standard metric. Since the conformal group of $(\S^m\times \H^d,\gr\oplus\ghyp)$ coincides with its isometry group, it follows that none of the infinitely many new metrics with constant scalar curvature are isometric to $\gr\oplus\ghyp$. However, some of these new metrics may be isometric to one another.
\end{remark}

As explained in the Introduction, Theorem~\ref{thm:first} is a consequence of Theorem~\ref{thm:general}, Example~\ref{ex:infiniteresfinite}, and \cite[Thm.\ A]{borel}.

\subsection{Singular Yamabe problem}\label{subsec:syp}
Given a closed manifold $(M,\g)$ and a closed subset $\Lambda\subset M$, the singular Yamabe problem consists of finding a complete metric $\g'$ on $M\setminus\Lambda$ that has constant scalar curvature and is conformal to $\g$. In other words, these are solutions to the Yamabe problem on $M$ that blow up on $\Lambda$. Consider the case in which $(M,\g)$ is the round sphere $(\S^m,\gr)$ and $\Lambda=\S^k$ is a round subsphere, which was also studied in \cite{bps-jdg,mazzeo91,mazzeo-smale,schoen87}. There is a conformal equivalence
\begin{equation}\label{eq:confss}
f\colon \big(\S^m\setminus \S^k,\gr\big)\to \big(\S^{m-k-1}\times\H^{k+1},\gr\oplus\ghyp\big),
\end{equation}
given by first using the stereographic projection with a point in $\S^k$ to obtain a conformal equivalence with $(\R^m\setminus\R^k,\gf)$, and second using cylindrical coordinates $\gf=\dd r^2+r^2\dd\theta^2+\dd y^2$ to conclude that $\tfrac{1}{r^2}\gf=\gr\oplus\ghyp$, see also \cite{bps-jdg,mazzeo-smale}.

The conformal equivalence \eqref{eq:confss} provides a trivial solution $f^*(\gr\oplus\ghyp)$ to the singular Yamabe problem on $\S^m\setminus \S^k$, with constant scalar curvature equal to $\scal_{m,k}=(m-2k-2)(m-1)$. If $k>(m-2)/2$, then $\scal_{m,k}<0$ and this is the unique solution by an argument involving the asymptotic maximum principle~\cite{mazzeo91}. Furthermore, this trivial solution is the unique \emph{periodic} solution if $k=(m-2)/2$, since in this case $\scal_{m,k}=0$ and any two conformal metrics with vanishing scalar curvature on a closed manifold are homothetic~\cite[p.\ 175]{aubin-book}. Thus, nonuniqueness of periodic solutions on $\S^m\setminus\S^k$ is only possible in the range $0\leq k<(m-2)/2$.
The existence of infinitely many periodic solutions on this entire range (Corollary~\ref{cor:syp}) follows from Theorem~\ref{thm:first} applied to $\S^{m-k-1}\times\H^{k+1}$ and $\S^{m-1}\times\R$, together with the conformal equivalences \eqref{eq:confss} and $\S^m\setminus\{\pm p\}\cong\R^m\setminus\{0\}\cong\S^{m-1}\times\R$.
Existence of infinitely many (nonisometric) periodic solutions if $\Lambda=\S^1$ is a great circle, i.e., $k=1$, $m\geq5$, was recently obtained using bifurcation techniques \cite{bps-jdg}.

By Remark~\ref{rem:nonisomthm}, none of these periodic solutions on $\S^m\setminus\S^k$ are isometric to the trivial solution. We conjecture that, furthermore, they are pairwise nonisometric.

\section{Multiplicity of solutions via collapse of flat manifolds}
\label{sec:bifcollapseflat}

In this section, we employ another method to obtain multiplicity of solutions to the Yamabe problem on noncompact product manifolds using bifurcation theory. This technique provides further information on the local structure of the space of solutions, and has been previously applied to the Yamabe problem in \cite{bp-calcvar,bp-pacific,bps-jdg,hector}.

\subsection{Flat manifolds} 
Let $(F,\h)$ be a closed flat manifold. It is well-known that $(F,\h)$ is isometric to the orbit space $\R^d/\pi$ of a free isometric action on $\R^d$ of a discrete cocompact group $\pi$, the fundamental group of $F$. Often, such groups are called \emph{Bieberbach groups}, and, accordingly, $F$ is called a \emph{Bieberbach manifold}. In what follows, for the convenience of the reader, we provide an overview of basic facts regarding such groups and manifolds; for more details see \cite{bdp,buser-bieber,charlap,szczepa-book,wolf-flat}.

Let $\aff=\GL(d)\ltimes\R^d$ be the group of affine transformations of $\R^d$, and $\Iso(\R^d)=\O(d)\ltimes\R^d$ be the subgroup of rigid motions. Elements of $\aff$ and $\Iso(\R^d)$ are denoted by $(A,v)$, with $A\in\GL(d)$ or $\O(d)$ and $v\in\R^d$; the group operation is $(A,v)\, (B,w)=(AB, Aw+v)$. The natural action of these groups on $\R^d$ is given by $(A,v)\cdot w=Aw+v$. We denote by
\begin{equation*}
\rot\colon\aff\longrightarrow\GL(d), \quad \rot(A,v)=A,
\end{equation*}
the projection homomorphism. Furthermore, given a subgroup $\pi\subset \Iso(\R^d)$, we denote by $\tra(\pi)$ the normal subgroup of \emph{pure translations} in $\pi$, defined as:
\begin{equation*}
\tra(\pi)=\pi\cap\ker(\rot).
\end{equation*}
Note that there is a short exact sequence
\begin{equation}
1\longrightarrow \tra(\pi)\longrightarrow \pi \longrightarrow \rot(\pi) \longrightarrow 1.
\end{equation}
A discrete subgroup $\pi\subset\Iso(\R^d)$ is called \emph{crystallographic} if it has compact fundamental domain in $\R^d$, so that $\R^d/\pi$ is a compact flat orbifold. A crystallographic group $\pi$ acts freely in $\R^d$ if and only if it is torsion-free, in which case $\R^d/\pi$ is a closed flat manifold, and $\pi$ is called a \emph{Bieberbach group}.

The most important facts about such groups are summarized by the following results of Bieberbach, which provided an answer to Hilbert's 18th problem:

\begin{bieber}[\sc Algebraic version]
The following hold:
\begin{enumerate}[\rm I.]
\item
If $\pi\subset\Iso(\R^d)$ is a crystallographic group, then $\rot(\pi)$ is finite and $\tra(\pi)$ is a lattice that spans $\R^d$.
\item
Let $\pi,\pi'\subset\Iso(\R^d)$ be crystallographic subgroups. If there exists an isomorphism $f\colon\pi\to\pi'$, then $f$ is a conjugation in $\aff$, i.e., there exists $\alpha\in\aff$ such that $f(\beta)=\alpha\beta\alpha^{-1}$ for all $\beta\in\pi$.
\item
For all $d$, there are only finitely many isomorphism classes of crystallographic subgroups of $\Iso(\R^d)$.
\end{enumerate}
\end{bieber}
The geometric interpretation of these statements in terms of Bieberbach manifolds $F=\R^d/\pi$ is as follows:

\begin{bieber}[\sc Geometric version]
The following hold:
\begin{enumerate}[\rm I.]
\item
If $(F,\h)$ is a closed flat manifold with $\dim F=d$, then $(F,\h)$ is covered by a flat torus of dimension $d$, and the covering map is a local isometry.
\item
If $F$ and $F'$ are closed flat manifolds of the same dimension with isomorphic fundamental groups, then $F$ and $F'$ are affinely equivalent.
\item
For all $d$, there are only finitely many affine equivalence classes of closed flat manifolds of dimension $d$.
\end{enumerate}
\end{bieber}

The torus covering $F=\R^d/\pi$ is given by $T^d=\R^d/\tra(\pi)$, and $\rot(\pi)\subset\O(d)$ is the holonomy group of $(F,\h)$. Since the holonomy of a Riemannian manifold depends only on its affine structure, it follows that any two flat metrics on a closed manifold $F$ have isomorphic holonomy groups.

\subsection{Moduli space of flat metrics}
The moduli space $\mathcal M_{\text{flat}}(F)$ of flat metrics on a closed manifold $F=\R^d/\pi$ can be determined from the algebraic data in~$\pi$, see \cite{bdp,wolf-flat} for details.

In what follows, given groups $H\subset G$, we denote by $\norm_G(H)$ and $\centr_G(H)$ the normalizer and centralizer of $H$ in $G$, respectively.

\begin{lemma}\label{thm:corAOn}
Two compact  flat $d$-manifolds $F=\R^d/\pi$ and $F'=\R^d/\pi'$ are isometric if and only if there is $(B,w)\in\Iso(\R^d)$ such that $(B,w)\,\pi\,(B,w)^{-1}=\pi'$. Moreover, if $\pi,\pi'\subset\Iso(\R^d)$ are isomorphic Bieberbach groups, i.e., $(A,v)\,\pi\,(A,v)^{-1}=\pi'$ for some $(A,v)\in\aff$, then $F=\R^d/\pi$ and $F'=\R^d/\pi'$ are isometric if and only if $A=B \, C$, with $B\in\O(d)$  and $C\in \mathcal N_\pi:=\rot\big(\!\norm_{\aff}(\pi)\big)\subset\GL(d,\R)$.
\end{lemma}

\begin{proof}
The first statement follows from lifting an isometry between $F$ and $F'$ to an isometry of $\R^d$.
Thus, if $\pi$ and $\pi'$ are isomorphic, then $F$ and $F'$ are isometric if and only if there exists $(B,w)\in\Iso(\R^d)$ such that $(B,w)\,\pi\,(B,w)^{-1}=(A,v)\,\pi\,(A,v)^{-1}$, i.e., $(C,z):=(B,w)^{-1}(A,v)\in\norm_{\aff}(\pi)$. If $F$ and $F'$ are isometric, then clearly $A=B \, C$, with $B\in\O(d)$  and $C\in\mathcal N_\pi$. Conversely, assume $A=B \, C$, with $B\in\O(d)$  and $C\in\mathcal N_\pi$. By definition, there exists $z\in\R^d$ such that $(C,z)\in\norm_{\aff}(\pi)$. Set $w=v-Bz$, so that  $(A,v)=(B,w)\, (C,z)$. Clearly, $(B,w)^{-1}(A,v)\in\norm_{\aff}(\pi)$, so $F$ and $F'$ are isometric.
\end{proof}

We associate to each closed flat manifold $F=\R^d/\pi$ the closed cone
\begin{equation}\label{eq:newXPhi}
\begin{aligned}
\mathcal C_F:&=\big\{A\in\GL(d,\R):A\,B\, A^{-1}\in\O(d) \mbox{ for all } B\in \rot(\pi)\big\}\\
&=\big\{A\in\GL(d,\R):A^{\rm t}A\in\centr_{\GL(d,\R)}\!\big(\rot(\pi)\big)\!\big\},
\end{aligned}
\end{equation}
where $A^{\rm t}$ is the transpose of $A$. It is easy to verify that $\mathcal C_F$ contains $\norm_{\GL(d,\R)}\!\big(\rot(\pi)\big)$, and if $A\in\mathcal C_F$, then $({A^t})^{-1}\in\mathcal C_F$. 
Moreover, left-multiplication defines an $\O(d)$-action on $\mathcal C_F$, and right-multiplication defines a $\norm_{\GL(d,\R)}\!\big(\rot(\pi)\big)$-action on $\mathcal C_F$.

Given a Bieberbach group $\pi\subset\Iso(\R^d)$, denote by $\h_\pi$ the flat metric on $F=\R^d/\pi$ for which the covering map $\R^d\to\R^d/\pi$ is Riemannian, i.e., a local isometry. The following characterization of the moduli space $\mathcal M_{\text{\rm flat}}(F)$ can be found in \cite{bdp,wolf-flat}.

\begin{proposition}\label{prop:flatmoduli}
For any flat metric $\h$ on $F$, there exists $A\in \mathcal C_F$ and $v\in\R^d$ such that $\h$ is isometric to $\h_{\pi'}$, where $\pi'=(A,v)\,\pi\,(A,v)^{-1}\subset\Iso(\R^d)$. Furthermore, $\h_{\pi}$ is isometric to $\h_{\pi'}$ if and only if $A=B\, C$, with $B\in\O(d)$ and $C\in\mathcal N_\pi$. Thus, the moduli space of flat metrics on $F$ is the double coset space $\mathcal M_{\text{\rm flat}}(F)\cong\O(d)\backslash\mathcal C_F/\mathcal N_\pi$.
\end{proposition}

\begin{proof}
The flat metric $\h$ must be of the form $\h_{\pi'}$ for some Bieberbach group $\pi'\subset\Iso(\R^d)$. Since $\R^d/\pi$ and $\R^d/\pi'$ are both diffeomorphic to $F$, $\pi$ and $\pi'$ must be isomorphic. By the Bieberbach Theorems II, there exists $A\in\mathcal C_F$ and $v\in\R^d$ such that $\pi'=(A,v)\,\pi\,(A,v)^{-1}$. By Lemma~\ref{thm:corAOn}, the metrics $\h_\pi$ and $\h_{\pi'}$ are isometric if and only if $A=B\, C$, with $B\in\O(d)$ and $C\in\mathcal N_\pi$, concluding the proof.
\end{proof}

\subsection{Collapse of flat manifolds}
Given a Bieberbach group $\pi\subset\Iso(\R^d)$, $A\in\mathcal C_F$ and $v\in\R^d$, if $\pi'=(A,v)\,\pi\,(A,v)^{-1}$, then clearly $\Vol(F,\h_{\pi'})=\det(A)\Vol(F,\h_\pi)$. We exploit this fact together with the above facts about $\mathcal M_{\text{\rm flat}}(F)$ to show that every closed flat manifold (of dimension $\geq2$) admits a $1$-parameter family of collapsing flat metrics; which implies it can be \emph{squeezed} just as the square torus $\R^2/\Z^2\cong\S^1(1)\times\S^1(1)$ can be squeezed through the family of flat tori $\S^1(t)\times\S^1(1/t)$, $t>0$.

\begin{proposition}\label{prop:squeeze}
Any closed flat manifold $(F,\h)$ has a real-analytic family $\h_t$ of flat metrics with $\h_1=\h$, $\Vol(F,\h_t)=\Vol(F,\h)$ and $\diam(F,\h_t)\nearrow+\infty$ as $t\searrow 0$.
\end{proposition}

\begin{proof}
By a result of Hiss and Szczepa\'nski~\cite{hiss-szczepa}, the holonomy representation of any closed flat manifold $F=\R^d/\pi$ is reducible. Let $E\subset\R^d$ be a nontrivial invariant subspace and let $E^\perp$ be its orthogonal complement, which is also invariant as the representation is orthogonal. Denote by $P$ and $P^\perp$ the orthogonal projections of $\mathds R^d$ onto $E$ and $E^\perp$ respectively. It is easy to see that, for all $t>0$, the linear maps
\begin{equation}\label{eq:Alambda}
A_t:=t^{\dim E-d}\cdot P+t^{\dim E}\cdot P^\perp\in\GL(d,\R)
\end{equation}
satisfy $A_t\in\mathcal C_F$ and $\det(A_t)=1$. Thus, the metrics $\h_t:=\h_{\pi_t}$ where $\pi_t=(A_t,0)\,\pi \,(A_t,0)^{-1}$, i.e., the metrics $\h_t\in\mathcal M_{\text{\rm flat}}(F)$ corresponding to the double coset of $A_t$, have fixed volume and arbitrarily large diameter as $t\searrow0$ (or $t\nearrow+\infty$).
\end{proof}

\subsection{Eigenvalues of the Laplacian}
All nonzero eigenvalues $\lambda_j(F,\h_t)$ of the Laplacian of the above collapsing family $(F,\h_t)$ of flat manifolds are nonconstant real-analytic functions of $t$.
Indeed, there is a Riemannian covering $\R^d/\tra(\pi_t)\to\R^d/\pi_t$, where $\tra(\pi_t)=A_t(\tra(\pi))$, and the spectrum of $\Delta_{\h_t}$ is contained in the spectrum of the Laplacian of the flat torus $\R^d/\tra(\pi_t)$.
The dual lattice to $\tra(\pi_t)$ is given by $\tra(\pi_t)^*=\big((A_t)^{\rm t}\big)^{-1}(\tra(\pi)^*)$, and by \cite[p.\ 146]{BerGauMAz} the eigenvalues of the Laplacian of $\R^d/\tra(\pi_t)$ are $4\pi\Vert x\Vert^2$, with $x\in \tra(\pi_t)^*$, which are nonconstant polynomials in $t$ and $\frac1t$, proving the above claim.

\begin{proposition}\label{prop:small-eigenvalues}
For any closed flat manifold $F$ and for all $\varepsilon>0$ and $j\in\N$, there exists a unit volume flat metric $\h$ on $F$ such that $\lambda_j(F,\h)<\varepsilon$.
\end{proposition}

\begin{proof}
By a classical estimate of Cheng~\cite[Cor.\ 2.2]{cheng75}, since $(F,\h)$ has $\Ric\geq0$ and $\dim F=d$, then
\begin{equation}\label{eq:cheng}
\lambda_j(F,\h)\le2j^2\,\frac{d(d+4)}{\diam(F,\h)^2}.
\end{equation}
By Proposition~\ref{prop:squeeze}, there are unit volume flat metrics $\h$ on $F$ for which the right-hand side of the above is arbitrarily small, which concludes the proof.
\end{proof}

\subsection{Bifurcation of solutions}
Let $M$ be a closed manifold and $\g_t$ be a $1$-parame\-ter family of unit volume constant scalar curvature metrics on $M$. We say that $t_*$ is a \emph{bifurcation instant} for $\g_t$ if there exist sequences 
$\{t_q\}_{q\in\N}$ converging to $t_*$ and $\{\phi_q\}_{q\in\N}$ of smooth nonconstant positive functions on $M$ such that:
\begin{enumerate}[(i)]
\item $\phi_q\to 1$ in the $C^{2,\alpha}$-topology;
\item $\g_q:=\phi_q\cdot \g_{t_q}$ is a unit volume constant scalar curvature metric on $M$.
\end{enumerate}

Applying standard variational bifurcation results to the Hilbert-Einstein functional on conformal classes of metrics, one obtains the following criterion for bifurcation of solutions to the Yamabe problem, see for instance \cite{LPZ12}.

\begin{theorem}\label{thm:abstrcrit}
Let $M$ be a closed manifold of dimension $n\ge3$ with a $1$-parameter family $\g_t$ of unit volume metrics with constant scalar curvature $\scal_{\g_t}$.
Let $\mathfrak i(M,\g_t)$ be the number of eigenvalues of $\Delta_{\g_t}$, counted with multiplicity, that are $<\frac{\scal_{\g_t}}{n-1}$.
Assume:
\begin{itemize}
\item[(a)] $\frac{\scal_{\g_t}}{n-1}\not\in\spec(\Delta_{\g_t})\setminus\{0\}$ for all $t\ne t_*$;
\item[(b)] $\mathfrak i(M,\g_{t_*-\varepsilon})\ne\mathfrak i(M,\g_{t_*+\varepsilon})$ for some small $\varepsilon>0$.
\end{itemize}
Then $t_*$ is a bifurcation instant for the family $\g_t$.
\end{theorem}

A discussion on the convergence of bifurcating branches of constant scalar curvature metrics $\phi_q\to 1$ can be found in \cite[Sec.~3]{hector}. As a consequence of Theorem~\ref{thm:abstrcrit}, we have the following criterion for bifurcation of products with flat manifolds:

\begin{corollary}\label{thm:corbifprodflat}
Let $(M,\g)$ be a closed unit volume Riemannian manifold with positive constant scalar curvature $\scal_\g$. Let $F$ be a closed flat manifold and $\h_t$, $t\in[t_*-\varepsilon,t_*+\varepsilon]$, be a $1$-parameter family of flat unit volume metrics on $F$. Set
\begin{equation}\label{eq:(k1k2)}
\mathfrak i_t:=\#\Big\{(j_1,j_2):j_1,j_2\ge0,\ \lambda_{j_1}(M,\g)+\lambda_{j_2}(F,\h_t)<\frac{\scal_\g}{\dim M+\dim F-1}\Big\},
\end{equation}
and assume:
\begin{itemize}
\item[(a)] for all $t\ne t_*$ and all $0\leq j\leq \mathfrak i(M,\g)-1$,
\begin{equation}\label{eq:cond(a)}
    \frac{\scal_\g}{\dim M+\dim F-1}-\lambda_j(M,\g)\not\in\spec(\Delta_{\h_t});
    \end{equation}
\item[(b)] $\mathfrak i_{t_*-\varepsilon}\ne\mathfrak i_{t_*+\varepsilon}$ for some small $\varepsilon>0$.
\end{itemize}
Then $t_*$ is a bifurcation instant for the family of metrics $\g\oplus \h_t$ on $M\times F$.
\end{corollary}

\begin{remark}
Note that the bifurcating branch in Corollary~\ref{thm:corbifprodflat} issuing from the family of product metrics $\g\oplus\h_t$ does not contain any other product metrics. Indeed, two product metrics are conformal if and only if they are homothetic.
\end{remark}

We are now establish the bifurcation result on closed product manifolds $M\times F$ which lies in the core of the multiplicity result for $M\times\R^d$ in Theorem~\ref{thm:flat}.

\begin{theorem}\label{thm:mainbif}
Let $(F,\h)$ be a closed flat manifold with unit volume. There exists a real-analytic family $\h_t$, $t\in\left(0,+\infty\right)$, of unit volume flat metrics on $F$, with $\h_1=\h$, such that, if $(M,\g)$ is a unit volume closed Riemannian manifold with constant positive scalar curvature, there is a discrete countable set of bifurcation instants for the family of metrics $\g\oplus \h_t$,  on $M\times F$.
\end{theorem}

\begin{proof}
The desired sequence of bifurcation instants is obtained by repeatedly applying Corollary~\ref{thm:corbifprodflat}. Let $\h_t$ be the family of metrics on $F$ given by Proposition~\ref{prop:squeeze}.

First, we claim that for all $\varrho>0$, the only possible accumulation points of
\begin{equation*}
\mathcal D_\varrho=\big\{t\in\left(0,+\infty\right):\varrho\in\spec\big(\Delta_{\g\oplus \h_t}\big)\big\}
\end{equation*}
are $0$ and $+\infty$. Indeed, the spectrum $\spec\big(\Delta_{\g\oplus \h_t}\big)$ is the set of eigenvalues
\begin{equation*}
\lambda_{j_1}(M,\g)+\lambda_{j_2}(F,\h_t),\qquad j_1,j_2\ge0,
\end{equation*}
and the functions $t\mapsto \lambda_j(F,\h_t)$ are polynomials in $t$ and $\frac1t$. Thus, for all fixed $\varrho$, $j_1$, and $j_2$, the set of $t$'s for which $\lambda_{j_1}(M,\g)+\lambda_{j_2}(F,\h_t)=\varrho$ is finite. Moreover, for each compact interval $[a,b]\subset (0,+\infty)$, there are only finitely many pairs $(j_1,j_2)$ such that $\lambda_{j_1}(M,\g)+\lambda_{j_2}(F,\h_t)=\varrho$ for some $t\in [a,b]$.
Therefore, $\mathcal D_\varrho\cap [a,b]$ is finite, proving the claim. 
Setting $\varrho=\frac{\scal_\g}{\dim M+\dim F-1}$, it follows that \eqref{eq:cond(a)} in assumption (a) of Corollary~\ref{thm:corbifprodflat} holds for all $t\in (0,+\infty)$ outside a locally finite set.

Second, we claim that $\mathfrak i_t\nearrow+\infty$ as $t\searrow 0$, where $\mathfrak i_t$ is defined in \eqref{eq:(k1k2)}.
Setting
\begin{equation*}
N_0:=\max\Big\{j\ge0:\lambda_j(M,\g)<\frac{\scal_\g}{\dim M+\dim F-1}\Big\},
\end{equation*}
it is easy to see that
\begin{equation}\label{eq:estimateindex}
\mathfrak i_t\geq \#\Big\{j\ge0:\lambda_j(F,\h_t)<\frac{\scal_\g}{\dim M+\dim F-1}-\lambda_{N_0}(M,\g)\Big\}.
\end{equation}
Since $\diam(F,\h_t)$ becomes arbitrarily large as $t\searrow 0$, we also have that
\begin{equation*}
\lim\limits_{t\searrow0}\diam\big(M\times F,\g\oplus \h_t\big)=\lim\limits_{t\searrow0}\sqrt{\diam(M,\g)^2+\diam(F, \h_t)^2}=+\infty.
\end{equation*}
Thus, by Cheng's eigenvalue estimate \eqref{eq:cheng}, see Proposition~\ref{prop:small-eigenvalues}, we find that the right-hand side of \eqref{eq:estimateindex} becomes unbounded as $t\searrow 0$.

Therefore, Corollary~\ref{thm:corbifprodflat} can be applied to an infinite sequence of sufficiently small $t_*\in\left(0,+\infty\right)$, yielding the desired sequence of bifurcation instants.
\end{proof}

Finally, Theorem~\ref{thm:flat} in the Introduction follows from Theorem~\ref{thm:mainbif} applied to the closed manifold $M\times F$, where $F=\R^d/\pi$, which can be assumed to have unit volume by to rescaling. Note that the pull-back to $M\times\R^d$ of metrics on $M\times F$ with constant scalar curvature that are conformal to $\g\oplus\h_t$ are $\pi$-periodic solutions to the Yamabe problem on $\big(M\times\R^d,\,\g\oplus\gf\big)$.

\end{document}